\documentclass{article}
\usepackage{amsmath,amsthm,amssymb}

\newtheorem{theorem}{Theorem}[section]
\newtheorem{lemma}[theorem]{Lemma}
\newtheorem{cor}[theorem]{Corollary}
\theoremstyle{definition}

\theoremstyle{remark}
\newtheorem{remark}[theorem]{Remark}

\numberwithin{equation}{section}

\newcommand\on{\operatorname}
\newcommand\Ric{\on{Ric}}

\begin{document}

\title{Gradient solitons on doubly warped product manifolds}

\author{Adara M. Blaga and Hakan M. Ta\c{s}tan}

\maketitle


\begin{abstract}
   Firstly we provide new characterizations for doubly warped prod\-uct manifolds. Then we consider several types of gradient solitons on them such as Riemann, Ricci, Yamabe and conformal and examine the effect of a gradient soliton on a doubly warped product to its factor manifolds. Finally we in\-ves\-ti\-gate the concircularly flat and conharmonically flat cases of doubly warped products.
\end{abstract}

\noindent {\bf Keywords:} {doubly warped product manifold, Riemann soliton, Ricci soliton, Yamabe soliton, quasi-Einstein manifold.}

\noindent {\bf MSC 2020:} {53C20, 53C25.}


\section{Introduction}

Doubly warped products \cite{Er} are natural generalizations of warped products introduced by Bishop and O'Neill \cite{Bi}, which play an important role in differential geometry, as well as in physics, especially in the theory of relativity.

On the other hand, the interest in studying different geometric properties of Riemann \cite{hi}, Ricci \cite{ha} and Yamabe solitons \cite{de} has lately considerably increased. Self-similar solutions to Riemann, Ricci and Yamabe flow, they generalize spaces of constant sectional curvature and Einstein manifolds.

In the present paper, after deducing new properties of a doubly warped product manifold, we study the effect of a gradient Riemann, gradient Ricci and gradient Yamabe soliton on a doubly warped product manifold to the factor manifolds. Moreover, we prove that if the doubly warped product manifold is concircularly flat, then its factors are Einstein manifolds \cite{Al}, and if it is conharmonically flat, then its factors are gradient $f$-almost Ricci solitons \cite{gom}.

\section{Preliminaries}

\subsection{Doubly warped product manifolds}

Let $(M_1,g_1)$ and $(M_2,g_2)$ be two Riemannian manifolds, and let $f_1$ and $f_2$ be two positive smooth functions defined on $M_1$ and $M_2$, respectively. Let $\pi_1$ and $\pi_2$ be the canonical projections from $M_1 \times M_2$ onto $M_1$ and $M_2$, respectively. Then the \emph{doubly warped product manifold} \cite{Er} $_{{f}_{2}} M_1 \!\times_{{f}_{1}}\!M_2$ is the product manifold $M_1 \times M_2$ equipped with metric $g$ defined by
\begin{equation}
\label{e0}
g=(f_2 \circ \pi_2)^2 \pi_{1}^{*}(g_1)+(f_1 \circ \pi_1)^2 \pi_{2}^{*}(g_2),
\end{equation}
where $\pi_{i}^{*}(g_i)$ is the pullback of $g_i$ via $\pi_{i}$, for $i\in\{1,2\}.$ Each function $f_i$ is called a \emph{warping function} of the doubly warped product $(_{{f}_{2}}\!M_1\!\times_{{f}_{1}}\!M_2, g)$. If one of the functions $f_1$ and $f_2$ is constant, then we get a \emph{warped product manifold} \cite{Bi}.
If both $f_1$ and $f_2$ are constant, then we obtain a \emph{direct product manifold} \cite{Chen}. A doubly warped product manifold is said to be \emph{non-trivial} if neither $f_1$ nor $f_2$ is constant.\\

Let $(_{{f}_{2}}\!M_1\!\times_{{f}_{1}}\!M_2,g)$ be a doubly warped product manifold with the Levi-Civita connection $\nabla$
and denote by $\nabla^{i}$ the Levi-Civita connection of $g_{i}$, for $i\in\{1,2\}.$ By usual convenience, we denote by $\mathfrak{L}(M_{i})$ the set of lifts of vector fields on $M_{i}$ and we will use the same notation for a vector field and for its lift. Also, we shall use the same notation for a metric and for its pullback. On the other hand, each $\pi_{i}$ is a (positive) homothety, so it preserves the Levi-Civita connection. Thus, there is no confusion using the same notation for a connection on $M_{i}$ and for its pullback via $\pi_{i}.$ Then, the covariant derivative formulas \cite{Er} for a doubly warped product manifold are given by:
\begin{align}
\nabla_{X}Y&=\nabla^{1}_{X}Y-g(X,Y)\nabla(\ln (f_{2}\circ \pi_2)),\label{e1}\\
\nabla_{X}V&=\nabla_{V}X=V(\ln (f_{2}\circ \pi_2))X+X(\ln (f_{1}\circ \pi_1))V,\label{e2}\\
\nabla_{U}V&=\nabla^{2}_{U}V-g(U,V)\nabla(\ln (f_{1}\circ\pi_1))\label{e3},
\end{align}
for $X, Y\in\mathfrak{L}(M_{1})$ and $U, V\in\mathfrak{L}(M_{2})$. It follows that $M_{1}\times\{p_{2}\}$ and $\{p_{1}\}\times M_{2}$
are totally umbilical submanifolds with closed mean curvature vector fields in $(_{{f}_{2}}\!M_1\!\times_{{f}_{1}}\!M_2,g)$ \cite{Ol},
where $p_{1}\in M_{1}$ and $p_{2}\in M_{2}.$ For more details on doubly warped products, we refer to the papers \cite{Er, GeTa, Ol, Un}.

\begin{remark}
From now on, we will put $k=\ln f_{1}$ (resp. $ l=\ln f_{2}$) and use the same symbol for the function $k$ (resp. $ l$) and its pullback
$k\circ\pi_{1}$ (resp. $l\circ\pi_{2}$).
\end{remark}

Now, let $\psi$ be a smooth function on a doubly warped product $(_{{f}_{2}}\!M_1\!\times_{{f}_{1}}\!M_2,g)$.
Then, for any $X\in\mathfrak{L}(M_{1})$ and $U\in\mathfrak{L}(M_{2})$, by the definition of the Hessian tensor and by using \eqref{e1} and \eqref{e2}, we have
\begin{equation}
\label{a4}
\begin{array}{c}
h^\psi(X,U)=X(k)U(\psi)-X(\psi)U(l).
\end{array}
\end{equation}
Next, we define $h^{\psi}_1(X,Y)=XY(\psi)-(\nabla^{1}_{X}Y)(\psi)$, for all $X,Y\in\mathfrak{L}(M_{1})$ and
$h^{\psi}_2(U,V)=UV(\psi)-(\nabla^{2}_{U}V)(\psi)$, for all $U,V\in\mathfrak{L}(M_{2}).$ By using \eqref{e1} and \eqref{e3},
the Hessian tensor $h^\psi$ of $\psi$ satisfies
\begin{equation}
\label{e4}
\begin{array}{c}
h^\psi(X,Y)=h^{\psi}_1(X,Y)+g(X,Y)g(\nabla l,\nabla\psi)
\end{array}
\end{equation}
and
\begin{equation}
\label{e5}
\begin{array}{c}
h^\psi(U,V)=h^{\psi}_2(U,V)+g(U,V)g(\nabla k,\nabla \psi).
\end{array}
\end{equation}
Since $\nabla k\in\mathfrak{L}(M_{1})$ and $\nabla l\in\mathfrak{L}(M_{2})$, from \eqref{e4} and \eqref{e5} we deduce that
\begin{align}
h^k(X,Y)&=h^k_1(X,Y),\label{e04}\\
h^l(X,Y)&=g(X,Y)g(\nabla l,\nabla l),\label{0e4}\\
h^k(U,V)&=g(U,V)g(\nabla k,\nabla k),\label{e05}\\
h^l(U,V)&=h^l_2(U,V).\label{0e5}
\end{align}

Let $^{1}\!R$ and $^{2}\!R$ be the lifts of Riemann curvature tensors of $(M_1,g_1)$ and $(M_2,g_2)$, respectively and let $R$ be the Riemann curvature tensor of
the doubly warped product $( _{{f}_{2}}\!M_1\!\times_{{f}_{1}}\!M_2,g).$ Then, by a direct computation and by using \eqref{e1}--\eqref{e3}, we have the following relations.
\begin{lemma} \label{L1}
Let $X,Y, Z\in\mathfrak{L}(M_{1})$ and $U,V,W\in\mathfrak{L}(M_{2})$. Then, we have
\begin{align}
R_{XY}Z&=^{1}\!R_{XY} Z+g(X,Z)H^l(Y)-g(Y,Z)H^l(X),\label{e6}\\
R_{XY}U&=U(l)\bigg(Y(k)X-X(k)Y\bigg),\label{e7}\\
R_{UV}X&=X(k)\bigg(V(l)U-U(l)V\bigg),\label{e8}\\
R_{XU}Y&=\bigg(h_1^k(X,Y)+X(k)Y(k)\bigg)U+Y(k)U(l)X\nonumber\\
\quad&\quad+g(X,Y)\bigg(H^l(U)+U(l)\nabla l\bigg),\label{e9}
\end{align}
\begin{align}
R_{UX}V&=\bigg(h_2^l(U,V)+U(l)V(l)\bigg)X+V(l)X(k)U\nonumber\\
\quad&\quad+g(U,V)\bigg(H^k(X)+X(k)\nabla k\bigg),\label{e10}\\
R_{UV}W&=^{2}\!R_{UV}W+g(U,W)H^k(V)-g(V,W)H^k(U),\label{e11}
\end{align}
where $H^k$ is the Hessian tensor of $k$ on $( _{{f}_{2}}\!M_1\!\times_{{f}_{1}}\!M_2,g)$, i.e., $H^k(E)=\nabla_E \nabla k$, for any vector field $E$ on $_{{f}_{2}}\!M_1 \!\times_{{f}_{1}}\!M_2.$
\end{lemma}

Let $^{1}\!\Ric$ and $^{2}\!\Ric$ be the lifts of Ricci curvature tensors of $(M_1,g_1)$ and $(M_2,g_2)$, respectively and let $\Ric$ be the Ricci curvature tensor of
the doubly warped product $( _{{f}_{2}}\!M_1\!\times_{{f}_{1}}\!M_2,g).$ Then, by a direct computation and by using \eqref{e6}--\eqref{e11} and \eqref{0e4}--\eqref{e05}, we have the following relations.
\begin{lemma} \label{L2}
Let $X,Y\in\mathfrak{L}(M_{1})$ and $U,V\in\mathfrak{L}(M_{2})$. Then, we have
\begin{align}
\Ric(X,Y)&=\,^{1}\!\Ric(X,Y)-\frac{m_2}{f_1}h_1^{f_1}(X,Y)-g(X,Y)\Delta l,\label{e12}\\
\Ric(X,U)&=(m_1+m_2-2)X(k)U(l),\label{e13}\\
\Ric(U,V)&=\,^{2}\!\Ric(U,V)-\frac{m_1}{f_2}h_2^{f_2}(U,V)-g(U,V)\Delta k,\label{e14}
\end{align}
where $\Delta$ is the Laplacian operator on $(_{{f}_{2}}\!M_1\!\times_{{f}_{1}}\!M_2,g)$ and $m_{i}=\dim(M_{i})$, for $i\in\{1,2\}$.
\end{lemma}

Let $^{1}\!Q$ and $^{2}\!Q$ be the lifts of Ricci operators of $(M_1,g_1)$ and $(M_2,g_2)$, respectively and let $Q$ be the Ricci operator of
the doubly warped product $( _{{f}_{2}}\!M_1\!\times_{{f}_{1}}\!M_2,g).$ Then, by a direct computation and by using \eqref{e12} and \eqref{e14}, we have the following relations.
\begin{lemma} \label{L5}
Let $X\in\mathfrak{L}(M_{1})$ and $U\in\mathfrak{L}(M_{2})$. Then, we have
\begin{equation}\label{e124}
\begin{array}{ll}
QX=\frac{1}{f_2^2}\left(\,^{1}\!QX-\frac{m_2}{f_1}H^{f_1}(X)-f_2^2\Delta l \cdot X\right),
\end{array}
\end{equation}
\begin{equation}\label{e144}
\begin{array}{ll}
QU=\frac{1}{f_1^2}\left(\,^{2}\!QU-\frac{m_1}{f_2}H^{f_2}(U)-f_1^2\Delta k \cdot U\right).
\end{array}
\end{equation}
\end{lemma}
\begin{remark} \label{re1}
Let $\{e_{1},\dots, e_{m_{1}}, \omega_{1},\dots, \omega_{m_{2}}\}$ be an orthonormal basis of the doubly warped product $(_{{f}_{2}}\!M_1\!\times_{f_1}\!M_2, g),$
where $\{e_{1},\dots,e_{m_{1}}\}$ are tangent to $M_1$ and $\{\omega_{1},\dots, \omega_{m_{2}}\}$ are tangent to $M_2$.
Then by \eqref{e0}, we see that $\{f_{2}e_{1},\dots, f_{2}e_{m_{1}}\}$ is an orthonormal basis of $(M_1, g_1)$ and $\{f_{1}\omega_{1},\dots, f_{1}\omega_{m_{2}}\}$ is an orthonormal basis of $(M_2, g_2).$
\end{remark}
Let $^{1}\!\tau$ and $^{2}\!\tau$ be the lifts of scalar curvatures of $(M_1,g_1)$ and $(M_2,g_2)$, respectively and let $\tau$ be the scalar curvature of
the doubly warped product $(_{{f}_{2}}\!M_1\!\times_{f_1}\!M_2, g).$ Then, by Lemma \ref{L2} and Remark \ref{re1}, we obtain
\begin{equation}
\label{scal0}
\begin{array}{c}
\tau=\displaystyle\frac{^{1}\!\tau}{f^{2}_{2}}+\displaystyle\frac{^{2}\!\tau}{f^{2}_{1}}-\displaystyle\frac{m_{2}}{f_1f^{2}_{2}}\Delta_{1} f_1-\displaystyle\frac{m_{1}}{f^{2}_{1}f_2}\Delta_{2} f_2-m_{1}\Delta l-m_{2}\Delta k,
\end{array}
\end{equation}
where $\Delta_{i}$ is the lift of the Laplacian operator on $(M_{i}, g_{i})$ and $m_{i}=\dim(M_{i})$, for $i\in\{1,2\}.$

\subsection{Gradient Yamabe, Gradient Ricci and Gradient Riemann Solitons}

A pseudo-Riemannian manifold $(M^m,g)$ is said to be a \emph{gradient Yamabe soliton} \cite{ch}
if there exists a smooth function $\psi$ on $M$ and a real constant $\lambda$ satisfying
\begin{equation}\label{e15}
\begin{array}{ll}
h^{\psi}=(\tau-\lambda)g,
\end{array}
\end{equation}
where $\tau$ denotes the scalar curvature of $(M,g)$.
More generally, if
\begin{equation}\label{e16}
h^{\psi}=\gamma g
\end{equation}
holds for some smooth function $\gamma$, then the triple $(M,g,\psi)$ is called \emph{conformal gradient soliton}.\\

A pseudo-Riemannian manifold $(M^m,g)$ is said to be a \textit{gradient Ricci soliton} \cite{ha} if there exists a smooth function $\psi$ on $M$ and a real constant $\lambda$ satisfying
\begin{equation}\label{e17}
h^{\psi}+\Ric=\lambda g,
\end{equation}
where $\Ric$ is the Ricci curvature of $(M,g)$.\\

A pseudo-Riemannian manifold $(M^m,g)$ is said to be a \emph{gradient Riemann soliton} \cite{hi} if there exists a smooth function $\psi$ on $M$ and a real constant $\lambda$ satisfying
\begin{equation}\label{riemann}
h^{\psi}\wedge g+R=\lambda G,
\end{equation}
where $R$ is the Riemann curvature of $(M,g)$, $G=\frac{1}{2}(g\wedge g)$ and $\wedge$ is the Kulkarni-Nomizu product.
Then for any vector field $X,Y,Z,W$ on $M$, equation \eqref{riemann} is explicitly expressed as
\begin{equation}\label{riemannequation}
R(X,Y,Z,W)+g(X,W)h^{\psi}(Y,Z)+g(Y,Z)h^{\psi}(X,W)
\end{equation}
$$-g(X,Z)h^{\psi}(Y,W)-g(Y,W)h^{\psi}(X,Z)=\lambda \left(g(X,W)g(Y,Z)-g(X,Z)g(Y,W)\right),$$
which by contraction over $X$ and $W$, gives
\begin{equation}\label{riemannlie}
h^{\psi}+\frac{1}{m-2}\Ric=\frac{(m-1)\lambda-\Delta \psi}{m-2}g,
\end{equation}
provided $m\geq 3$. If $m=2$, then $\Ric=(\lambda-\Delta \psi)g$.\\

Generalizing the notions of Yamabe and Ricci soliton, we talk about $\eta$-Yamabe and $\eta$-Ricci soliton.\\

A pseudo-Riemannian manifold $(M^m,g)$ is said to be a \emph{gradient $\eta$-Yamabe soliton} \cite{de}
if there exist smooth functions $\psi$, $\lambda$ and $\mu$ on $M$ satisfying
\begin{equation}\label{e06}
h^{\psi}=(\tau-\lambda)g+\mu \eta\otimes \eta,
\end{equation}
where $\tau$ denotes the scalar curvature of $(M,g)$ and $\eta$ is a $1$-form on $M$.\\

A pseudo-Riemannian manifold $(M^m,g)$ is said to be a \emph{gradient $\eta$-Ricci soliton} \cite{cho}
if there exist smooth functions $\psi$, $\lambda$ and $\mu$ on $M$ satisfying
\begin{equation}\label{e07}
h^{\psi}+\Ric=\lambda g+\mu \eta\otimes \eta,
\end{equation}
where $\eta$ is a $1$-form on $M$.\\

In all of the above cases, if $\lambda > 0, \lambda < 0$ or $\lambda = 0$, then the soliton is called a shrinking, expanding or steady, respectively.
If $\lambda$ is allowed be a smooth function on $M$, then $(M,g,\psi)$ is called a \emph{gradient almost Yamabe}, a \emph{gradient almost Ricci},
a \emph{gradient almost Riemann}, a \textit{gradient almost $\eta$-Yamabe} and a \textit{gradient almost $\eta$-Ricci soliton}, respectively.\\

Another generalization for the notion of Ricci soliton was introduced in \cite{gom}. \\

A pseudo-Riemannian manifold $(M^m,g)$ is said to be a \emph{gradient $f$-almost Ricci soliton} \cite{gom}
if there exist smooth functions $f$, $\psi$ and $\lambda$ on $M$ satisfying
\begin{equation}\nonumber
f h^{\psi}+\Ric=\lambda g.
\end{equation}

We generalize this notion to gradient $f$-almost $\eta$-Ricci soliton as follows.\\

A pseudo-Riemannian manifold $(M^m,g)$ is said to be a \emph{gradient $f$-almost $\eta$-Ricci soliton}
if there exist smooth functions $f$, $\psi$, $\lambda$ and $\mu$ on $M$ satisfying
\begin{equation}\nonumber
f h^{\psi}+\Ric=\lambda g+\mu \eta\otimes \eta,
\end{equation}
where $\eta$ is a $1$-form on $M$.

A Riemannian manifold $(M^m,g)$, $m\geq 2$, is said to be an \emph{Einstein manifold} \cite{Al} if its Ricci tensor $\Ric$ satisfies the condition $\Ric=\frac{\tau}{m}g$, where $\tau$ denotes the \emph{scalar curvature} of $(M,g)$. A non-flat Riemannian manifold $(M,g)$, $m\geq 2$, is said to be a \emph{quasi-Einstein} \cite{Cha} if the condition
\begin{equation}
\label{e20}
\begin{array}{c}
\Ric=\alpha g+\beta A\otimes A
\end{array}
\end{equation}
is fulfilled on $M$, where $\alpha$ and $\beta$ are scalar functions on $M$ with $\beta \neq 0$ and $A$ is non-zero $1$-form such that
$g(X,\xi)=A(X)$,
for every vector field $X$ on $M$, $\xi$ being a unit vector field which is called the generator of the manifold $M.$ If $\beta=0$, then the manifold reduces to an Einstein manifold.

\section{Main Results}

We first give some characterization for doubly warped product manifolds.\\

Let $(M=_{{f}_{2}}\!M_1 \!\times_{{f}_{1}}\!M_2, g)$ be a doubly warped product manifold. Then, by using
\eqref{e04}--\eqref{0e5} and Remark \ref{re1}, we have
\begin{equation}
\label{b1}
\Delta k=\displaystyle\frac{{1}}{f^{2}_{2}}\Delta_{1}k+m_{2}f^{2}_{2}g_{1}(\nabla k, \nabla k)
\end{equation}
and
\begin{equation}
\label{b2}
\Delta l=\displaystyle\frac{{1}}{f^{2}_{1}}\Delta_{2}l+m_{1}f^{2}_{1}g_{2}(\nabla l, \nabla l).
\end{equation}
We now suppose that the first factor manifold $M_{1}$ is compact and the warping function $k$
is harmonic with respect to the Laplacian $\Delta$. Then we have
\begin{equation}
\label{b3}
\Delta_{1} k=-m_{2}f^{4}_{2}g_{1}(\nabla k, \nabla k),
\end{equation}
from \eqref{b1}, which says that $\Delta_{1} k$ has constant sign. More precisely, we find
$\Delta_{1} k\leq\nolinebreak 0$ on $M_{1}$. Thus, by Hopf's lemma, we conclude that $f_{1}$ is a constant function (say
$f_{1}=c_{1}=constant$) on $M_{1}$, since $M_{1}$ is compact. Hence, we can write $g=\nolinebreak f^{2}_{2}g_1\oplus \tilde{g_{2}}$,
where $\tilde{g_{2}}=c^{2}_{1} g_2$. Namely, ${}_{{f}_{2}}\!M_1\!\times_{{f}_{1}}\!M_2$
can be expressed as a warped product $_{{f}_{2}}\!M_1\!\times M_2$ with warping function $f_2$,
where the metric tensor of $M_2$ is $\tilde{g_{2}}$ given above. Therefore, we can state
\begin{theorem} \label{mnth1}
Let $(M=_{{f}_{2}}\!M_1 \!\times_{{f}_{1}}\!M_2, g)$ be a doubly warped product manifold with
the first factor manifold $M_{1}$ compact. If the warping function $k$
is harmonic with respect to the Laplacian $\Delta,$ then $M$ is a
warped product manifold of the form $_{{f}_{2}}\!M_1\!\times M_2$.
\end{theorem}
By using \eqref{b2}, we can prove the following result in a similar way.
\begin{theorem} \label{mnth2}
Let $(M=_{{f}_{2}}\!M_1 \!\times_{{f}_{1}}\!M_2, g)$ be a doubly warped product manifold with
the second factor manifold $M_{2}$ compact. If the warping function $l$
is harmonic with respect to the Laplacian $\Delta$, then $M$ is a
warped product manifold of the form  $M_1 \!\times_{{f}_{1}}\!M_2.$
\end{theorem}
From Theorems \ref{mnth1} and \ref{mnth2}, we get the following result.
\begin{theorem} \label{mnth3}
Let $(M=_{{f}_{2}}\!M_1 \!\times_{{f}_{1}}\!M_2, g)$ be a compact doubly warped product manifold.
If the warping functions $k$ and $l$ are harmonic with respect to the Laplacian $\Delta,$
then $M$ is a usual product manifold of the form $M_1 \times M_2.$
\end{theorem}
Next, we characterize gradient Yamabe solitons on doubly warped product structures.
\begin{theorem} \label{mth1}
Let $(M=_{{f}_{2}}\!M_1 \!\times_{{f}_{1}}\!M_2, g)$ be a doubly warped product manifold. If $(M, g)$ is a gradient Yamabe soliton
whose potential function $\psi$ only depends on the points of $M_{1},$ then it is a
warped product manifold of the form $M_1 \!\times_{{f}_{1}}\!M_2.$
\end{theorem}
\begin{proof} Under the given conditions in the hypothesis, for any $X\in\mathfrak{L}(M_{1})$ and $U\in\mathfrak{L}(M_{2}),$ we have
\begin{equation}\nonumber
X(k)U(\psi)-X(\psi)U(l)=0,
\end{equation}
from \eqref{a4} and \eqref{e15}. We know $U(\psi)=0,$ since $\psi$ only depends on the points of $M_{1}$, so, we obtain
\begin{equation}\nonumber
X(\psi)U(l)=0.
\end{equation}
Hence, we get $U(l)=0$, for all $U\in\mathfrak{L}(M_{2}).$ Then, we find $l=constant,$ so $f_{2}=c_{2}$ for some constant $c_{2}.$
Thus, we can write $g=\tilde{g_1}\oplus f^{2}_{1}g_{2}$, where $\tilde{g_{1}}=c^{2}_{2} g_1$, that is, ${}_{{f}_{2}}\!M_1\!\times_{{f}_{1}}\!M_2$ can be expressed as a warped product $M_1\!\times _{{f}_{1}}\!M_2$ with warping function $f_1$, where the metric tensor of $M_1$ is $\tilde{g_{1}}$ given above.
This proves the assertion, as desired.
\end{proof}
In a similar way, we have the following result.
\begin{theorem} \label{mth2}
Let $(M=_{{f}_{2}}\!M_1 \!\times_{{f}_{1}}\!M_2, g)$ be a doubly warped product manifold. If $(M, g)$ is a gradient Yamabe soliton
whose potential function $\psi$ only depends on the points of $M_{2},$ then it is a
warped product manifold of the form $_{{f}_{2}}\!M_1 \!\times M_2.$
\end{theorem}
By Theorems \ref{mth1} and \ref{mth2}, we obtain the following result.
\begin{cor} \label{cor1}
Let $(M=_{{f}_{2}}\!M_1 \!\times_{{f}_{1}}\!M_2, g)$ be a non-trivial doubly warped product manifold. If $(M, g, \psi)$ is a gradient Yamabe soliton,
then the potential function $\psi$ has to depend both on the points of $M_{1}$ and of $M_{2}.$
\end{cor}
Now, we aim to investigate the geometry of the factor manifolds of a gradient Yamabe soliton with doubly warped product structure.
\begin{theorem} \label{mtheo1}
Let $(M=_{{f}_{2}}\!M_1 \!\times_{{f}_{1}}\!M_2, g)$ be a doubly warped product manifold. Then $(M,g,\psi)$ is a gradient Yamabe soliton if and only if the following statements hold:\\

\item \textbf{(a)} \quad $(M_1,g_1,\psi_{1})$ is a gradient almost Yamabe soliton; \\

\item \textbf{(b)} \quad $(M_2,g_2,\psi_{2})$ is a gradient almost Yamabe soliton;\\

\item \textbf{(c)} \quad $X(k)U(\psi)-X(\psi)U(l)=0$ for $X\in\mathfrak{L}(M_{1})$ and $U\in\mathfrak{L}(M_{2}),$\\
\\
where $\psi_{i}=\psi|_{M_{i}}$, for $i\in\{1,2\}.$
\end{theorem}
\begin{proof} Let $(M=_{{f}_{2}}\!M_1 \!\times_{{f}_{1}}\!M_2,g)$ be a gradient Yamabe soliton with the potential function $\psi$.
Then, we have
\begin{equation}\nonumber
h^{\psi}=(\tau-\lambda)g,
\end{equation}
from \eqref{e15}. Hence, using \eqref{e0}, \eqref{e4} and \eqref{scal0}, we obtain
\begin{equation}\nonumber
h_{1}^{\psi_{1}}=(^{1}\tau-\lambda_{1})g_{1}
\end{equation}
on $M_1,$ where
\begin{equation}
\nonumber
\lambda_{1}=-\displaystyle\frac{f^{2}_{2}}{f^{2}_{1}}\,^{2}\!\tau+f_2^2\big(\lambda+g(\nabla l,\nabla \psi)+m_1 \Delta l+m_2 \Delta k\big)+\frac{1}{f_1^2}(m_2f_1\Delta_1f_1+m_1f_2\Delta_2f_2),
\end{equation}
which means that $(M_1,g_1,\psi_{1})$ is a gradient almost Yamabe soliton, as desired.
The assertion \textbf{(b)} can be obtained in a similar way. On the other hand,
for $X\in\mathfrak{L}(M_{1})$ and $U\in\mathfrak{L}(M_{2})$, using \eqref{a4} and \eqref{e15}, we easily get
assertion \textbf{(c)}. The converse is just a verification.
\end{proof}
Next, we give a characterization for a non-trivial doubly warped product.
\begin{theorem} \label{mthm1}
A non-trivial doubly warped product manifold $(M=_{{f}_{2}}\!M_1 \!\times_{{f}_{1}}\!M_2, g)$
does not admit a conformal gradient soliton.
\end{theorem}
\begin{proof} Let $(M=_{{f}_{2}}\!M_1\!\times_{{f}_{1}}\!M_2, g)$ be a non-trivial doubly warped product manifold.
Assume that $(g,\psi)$ is a conformal gradient soliton on the non-trivial doubly warped product manifold $(M,g)$.
Then, for any vector fields $\bar{X}, \bar{Y}$ on $M$, we have
\begin{equation} \label{ne1}
h^{\psi}(\bar{X}, \bar{Y})=\gamma g(\bar{X}, \bar{Y}),
\end{equation}
from \eqref{e16}, where $\gamma$ is a smooth function on $M$.
By \eqref{ne1} and Lemma 4.1 of \cite{Che1}, we deduce that $\nabla\psi$ is a concircular vector field on $(M, g).$
In which case, it follows that $(M, g)$  is locally a warped product of the form $I\times_{\varphi}F$
from  Theorem 3.1 of \cite{Che1}, where $\varphi$ is a nowhere vanishing smooth function on an open interval $I$ of the real line and
$F$ is an $(m-1)-$dimensional Riemannian manifold, which is a contradiction.
\end{proof}

\begin{theorem} \label{mtheo2}
Let $(M=_{{f}_{2}}\!M_1 \!\times_{{f}_{1}}\!M_2,g)$ be a doubly warped product manifold. Then $(M,g,\psi)$ is a gradient Ricci soliton
with potential function $\psi$ and constant $\lambda$ if and only if the following statements hold:\\

\item \textbf{(a)} \quad $(\nabla^{1}\varphi_{1}, \lambda_{1}, \mu_{1})$ defines a gradient almost $\eta$-Ricci soliton on $(M_1, g_1)$,\\
\\
where $\varphi_{1}=\psi_{1}-m_{2}k,\, \lambda_{1}=f^{2}_{2}(\lambda+\Delta l-g(\nabla l,\nabla\psi))$, $\psi_{1}=\psi|_{M_{1}}$, $\eta=dk$ and $\mu_{1}=m_{2}$;\\
\\
\item \textbf{(b)}\quad $(\nabla^{2}\varphi_{2}, \lambda_{2}, \mu_{2})$ defines a gradient almost $\eta$-Ricci soliton on $(M_2, g_2)$,\\
\\
where $\varphi_{2}=\psi_{2}-m_{1}l,\, \lambda_{2}=f^{2}_{1}(\lambda+\Delta k-g(\nabla k,\nabla\psi))$, $\psi_{2}=\psi|_{M_{2}}$, $\eta=dl$ and $\mu_{2}=m_{1}$; \\

\item \textbf{(c)} \quad$(m_1+m_2-2)X(k)U(l)=X(\psi)U(l)-X(k)U(\psi)$, for $X\in\mathfrak{L}(M_{1})$ and $U\in\mathfrak{L}(M_{2}).$
\end{theorem}
\begin{proof} Let $(M=_{{f}_{2}}\!M_1 \!\times_{{f}_{1}}\!M_2,g)$ be a gradient Ricci soliton with the potential function $\psi$ and constant $\lambda$.
Then, we have
\begin{equation}\nonumber
\Ric+h^{\psi}=\lambda g,
\end{equation}
from \eqref{e17}. Hence, using \eqref{e0} and \eqref{e12}, we obtain
\begin{equation} \nonumber
^1\!\Ric+h_{1}^{\psi_{1}}=\lambda_{1}g_1+\frac{m_2}{f_1}h_{1}^{f_1}
\end{equation}
on $M_{1}$, where $\lambda_{1}=f^{2}_{2}(\lambda+\Delta l-g(\nabla l,\nabla\psi))$ and $\psi_{1}=\psi|_{M_{1}}.$ By using the fact that
$\frac{1}{f_1}h_{1}^{f_1}=h_{1}^{\ln f_1}+\frac{1}{f_{1}^{2}}df_{1}\otimes df_{1}$, we obtain
\begin{equation}\nonumber
^1\!\Ric+h_{1}^{\psi_{1}}=\lambda_{1}g_1+m_2h_{1}^{k}+m_2dk\otimes dk.
\end{equation}
By direct computations, we get
\begin{equation} \nonumber
^1\!\Ric+h_{1}^{\varphi_{1}}=\lambda_{1}g_1+m_2dk\otimes dk,
\end{equation}
where $\varphi_{1}=\psi_{1}-m_{2}k$. Thus, $(\nabla^{1}\varphi_{1}, \lambda_{1}, \mu_{1})$ defines a gradient almost $\eta$-Ricci soliton on $(M_1, g_1)$
as desired, with $\eta=dk.$
The assertion \textbf{(b)} can be obtained in a similar way. On the other hand,
for $X\in\mathfrak{L}(M_{1})$ and $U\in\mathfrak{L}(M_{2})$, we know that $g(X,U)=0.$
Thus, the assertion \textbf{(c)} follows immediately from \eqref{a4} and \eqref{e13}.
The converse is just a verification.
\end{proof}

\begin{theorem} \label{mtheo5}
Let $(M=_{{f}_{2}}\!M_1 \!\times_{{f}_{1}}\!M_2,g)$ be a doubly warped product manifold. If $(M,g,\psi)$ is a gradient Riemann soliton, then the following statements hold:\\

\item \textbf{(a)}\quad $(\nabla^{1}\varphi_{1}, \lambda_{1}, \mu_{1})$ defines a gradient almost $\eta$-Ricci soliton on $(M_1, g_1)$,\\
\\
where $\varphi_{1}=(m-2)\psi_{1}-m_{2}k,\, \lambda_{1}=f^{2}_{2}\big((m-1)\lambda+\Delta l-\Delta \psi -(m-2)g(\nabla l,\nabla\psi)\big)$,\\
\\
$\psi_{1}=\psi|_{M_{1}}$, $\eta=dk$ and $\mu_{1}=m_{2}$;\\

\item \textbf{(b)}\quad $(\nabla^{2}\varphi_{2}, \lambda_{2}, \mu_{2})$ defines a gradient almost $\eta$-Ricci soliton on $(M_2, g_2)$,\\
\\
where $\varphi_{2}=(m-2)\psi_{2}-m_{1}k,\, \lambda_{2}=f^{2}_{1}\big((m-1)\lambda+\Delta k-\Delta \psi -(m-2)g(\nabla k,\nabla\psi)\big)$,\\
\\
$\psi_{2}=\psi|_{M_{2}}$, $\eta=dl$ and $\mu_{2}=m_{1}.$
\end{theorem}
\begin{proof} Let $(M=_{{f}_{2}}\!M_1 \!\times_{{f}_{1}}\!M_2,g)$ be a gradient Riemann soliton with the potential function $\psi$.
Then, we have
\begin{equation}\nonumber
^{1}\!\Ric+(m-2)h^{\psi}=\left((m-1)\lambda-\Delta \psi\right) g,
\end{equation}
from \eqref{riemannlie}. Hence, using \eqref{e0} and \eqref{e12}, we obtain
\begin{equation}\nonumber
^{1}\!\Ric+(m-2)h_{1}^{\psi_{1}}=\lambda_{1}g_{1}+\frac{m_{2}}{f_{1}}h_{1}^{f_{1}}
\end{equation}
on $M_{1}$, where $\lambda_{1}=f^{2}_{2}\big((m-1)\lambda+\Delta l-\Delta \psi -(m-2)g(\nabla l,\nabla\psi)\big)$
and $\psi_{1}=\psi|_{M_{1}}.$ After some computations, we get
\begin{equation}\nonumber
^{1}\!\Ric+h_{1}^{\varphi_{1}}=\lambda_{1}g_{1}+m_{2}dk\otimes dk,
\end{equation}
where $\varphi_{1}=(m-2)\psi_{1}-m_{2}k.$ Thus, $(\nabla^{1}\varphi_{1}, \lambda_{1}, \mu_{1})$ defines a
gradient almost $\eta$-Ricci soliton on $(M_1, g_1)$ as desired, with $\eta=dk.$
The other assertion can be obtained in a similar way.
\end{proof}
\begin{theorem}\label{mtheo3}
Let $(M=_{{f}_{2}}\!M_1\!\times_{{f}_{1}}\!M_2,g)$ be a doubly warped product quasi-Einstein manifold with associated scalar functions $\alpha_0$ and $\beta_0$. Then \\
\item \textbf{(a)} \quad $(\nabla^{1}f_{1}, \lambda_{1}, \mu_1)$ defines a gradient $f$-almost $\eta$-Ricci soliton on $(M_1, g_1)$,\\
\\
where $f=-\frac{m_2}{f_1},\, \lambda_{1}=f^{2}_{2}(\alpha_0+\Delta l)$, $\eta=\tilde{A_{1}}$ and $\mu_{1}=\beta_0$;\\
\\
\item \textbf{(b)}\quad $(\nabla^{2}f_{2}, \lambda_{2}, \mu_2)$ defines a gradient $f$-almost $\eta$-Ricci soliton on $(M_2, g_2)$,\\
\\
where $f=-\frac{m_1}{f_2},\, \lambda_{2}=f^{2}_{1}(\alpha_0+\Delta k)$, $\eta=\tilde{A_{2}}$ and $\mu_{2}=\beta_0,$\\
\\
and $\tilde{A_{i}}=A|_{M_i},$ for $i\in\{1,2\}$.
\end{theorem}
\begin{proof} For any $X,Y\in\mathfrak{L}(M_{1})$, using \eqref{e12}, we have
\begin{equation}\nonumber
\alpha_0 g(X,Y)+\beta_0 A(X)A(Y)=\,^{1}\!\Ric(X,Y)-\frac{m_2}{f_1}h_1^{f_1}(X,Y)-g(X,Y)\Delta l,
\end{equation}
from \eqref{e20}. By using \eqref{e0}, we obtain
\begin{equation}\nonumber
^1\!\Ric(X,Y)=\,f^{2}_{2}(\alpha_0+ \Delta l)g_1(X,Y)+\beta_0 \tilde{A_{1}}(X)\tilde{A_{1}}(Y)
+\frac{m_2}{f_1}h_1^{f_1}(X,Y),
\end{equation}
which means that $(M_1, g_1)$ is a gradient $f$-almost $\eta$-Ricci soliton, as desired. The other assertion can be obtained in a similar way.
\end{proof}

\section{Concircularly Flat Doubly Warped Product Manifolds}

Let $(M, g)$ be an $m-$dimensional Riemannian manifold. The\emph{ concircular cur\-va\-ture tensor} $\mathcal{C}$ \cite{Ya} is defined by
\begin{equation}
\mathcal{C}=R-\displaystyle\frac{\tau}{m(m-1)}G,
\end{equation}
where $R$ is the Riemann curvature, $G=\frac{1}{2}(g\wedge g)$ and $\wedge$ is the Kulkarni-Nomizu product. If the concircular curvature $\mathcal{C}$ vanishes identically  on $M$, then $M$ is called \emph{concircularly flat}.\\

By using Lemma \ref{L1}, we obtain the components of the concircular curvature tensor on a doubly warped product as follows.
\begin{lemma} \label{L3}
Let $(M=_{{f}_{2}}\! M_1\!\times_{{f}_{1}}\!M_2,g)$ be a doubly warped product manifold. Then, for any $X,Y, Z\in\mathfrak{L}(M_{1})$ and $U,V,W\in\mathfrak{L}(M_{2}),$
the components of the concircular curvature tensor $\mathcal{C}$ of $M$ are given by
\begin{align}
\mathcal{C}_{XY}Z&=^{1}\!R_{XY} Z+g(X,Z)H^l(Y)-g(Y,Z)H^l(X)\nonumber\\
\quad&\quad-\displaystyle\frac{f^{2}_{2}\tau}{2m(m-1)}\big(g_{1}(Y, Z)X-g_{1}(X, Z)Y\big),\label{e23}\\
\mathcal{C}_{XY}U&=U(l)\bigg(Y(k)X-X(k)Y\bigg),\label{e24}
\end{align}
\begin{align}
\mathcal{C}_{UV}X&=X(k)\bigg(V(l)U-U(l)V\bigg),\label{e25}\\
\mathcal{C}_{XU}Y&=\bigg(h_1^k(X,Y)+X(k)Y(k)\bigg)U+Y(k)U(l)X\nonumber\\
\quad&\quad+g(X,Y)\bigg(H^l(U)+U(l)\nabla l\bigg)+\displaystyle\frac{f^{2}_{2}\tau}{2m(m-1)}g_{1}(X, Y)U,\label{e26}\\
\mathcal{C}_{UX}V&=\bigg(h_2^l(U,V)+U(l)V(l)\bigg)X+V(l)X(k)U\nonumber\\
\quad&\quad+g(U,V)\bigg(H^k(X)+X(k)\nabla k\bigg)+\displaystyle\frac{f^{2}_{1}\tau}{2m(m-1)}g_{2}(U, V)X,\label{e27}\\
\mathcal{C}_{UV}W&=^{2}\!R_{UV}W+g(U,W)H^k(V)-g(V,W)H^k(U)\nonumber\\
\quad&\quad-\displaystyle\frac{f^{2}_{1}\tau}{2m(m-1)}\big(g_{2}(V, W)U-g_{2}(U, W)V\big).\label{e28}
\end{align}
\end{lemma}

\begin{theorem}\label{mtheo4}
Let $(M=_{{f}_{2}}\!M_1\!\times_{{f}_{1}}\!M_2,g)$ be a doubly warped product manifold with $m_1,m_2>1$.
If $M$ is concircularly flat, then \\
\item \textbf{(a)} \quad $(M_1,g_1)$ is Einstein  with $^1\!\Ric=\mu_{1}g_{1},$\\
where $\mu_{1}=f^{2}_{2}(1-m_{1})\bigg(g(\nabla l, \nabla l)+\displaystyle\frac{\tau}{2m(m-1)}\bigg)$;\\

\item \textbf{(b)} \quad $(M_2,g_2)$ is Einstein  with $^2\!\Ric=\mu_{2}g_{2},$\\
where $\mu_{2}=f^{2}_{1}(1-m_{2})\bigg(g(\nabla k, \nabla k)+\displaystyle\frac{\tau}{2m(m-1)}\bigg)$;\\

\item \textbf{(c)} \quad $M$ is a warped product manifold of the form $M_1\!\times_{{f}_{1}} M_2$ or $G_{XY}\nabla k=0$; \\

\item \textbf{(d)} \quad $M$ is a warped product manifold of the form $_{{f}_{2}}M_1\!\times M_2$ or $G_{UV}\nabla l=0$.
\end{theorem}
\begin{proof} Let $(M=_{{f}_{2}}\!M_1\!\times_{{f}_{1}}\!M_2,g)$ be a concircularly flat doubly warped product manifold with $m_1,m_2>1$.
Then, for any $X,Y, Z\in\mathfrak{L}(M_{1})$, we have
\begin{equation}
\nonumber
^{1}\!R_{XY} Z=g(X,Z)H^l(Y)-g(Y,Z)H^l(X)-\displaystyle\frac{f^{2}_{2}\tau}{2m(m-1)}\big(g_{1}(Y, Z)X-g_{1}(X, Z)Y\big),
\end{equation}
from \eqref{e23}. Hence, for any $T \in\mathfrak{L}(M_{1}),$ using \eqref{e0} and \eqref{0e4}
$$g_{1}(H^l(X), T)=g(\nabla l, \nabla l)g_{1}(X, T),$$
we obtain
\begin{equation}
\nonumber
\begin{array}{c}
g_{1}(^{1}\!R_{XY}Z, T) =f^{2}_{2}g(\nabla l, \nabla l)g_{1}(X,Z)g_{1}(Y, T)
-f^{2}_{2}g(\nabla l, \nabla l)g_{1}(Y,Z)g_{1}(X, T)\\
\qquad-\displaystyle\frac{f^{2}_{2}\tau}{2m(m-1)}\big(g_{1}(Y, Z)g_{1}(X, T)-g_{1}(X, Z)g_{1}(Y, T)\big).
\end{array}
\end{equation}
By contracting this equation over $X$ and $T$, we get
\begin{equation}
\nonumber
^1\!\Ric(Y, Z)=f^{2}_{2}(1-m_{1})\bigg(g(\nabla l, \nabla l)+\displaystyle\frac{\tau}{2m(m-1)}\bigg)g_{1}(Y, Z),
\end{equation}
which proves \textbf{(a)}. The assertion \textbf{(b)} can be obtained in a similar way.
On the other hand, for any $X,Y\in\mathfrak{L}(M_{1})$ and $U\in\mathfrak{L}(M_{2})$, we have
\begin{equation}
\nonumber
U(l)\bigg(Y(k)X-X(k)Y\bigg)=0,
\end{equation}
from \eqref{e24}. It follows that $U(l)=0$, for all $U\in\mathfrak{L}(M_{2}),$ or $Y(k)X-X(k)Y=0$, for all $X,Y\in\mathfrak{L}(M_{1})$.
The last equation is equivalent to the equation $G_{XY}\nabla k=0.$ In the case of $U(l)=0$,
we see that $M$ is a warped product manifold of the form  $M_1\!\times_{{f}_{1}} M_2$
as in the proof of Theorem \ref{mth1}. The last assertion \textbf{(d)} can be obtained in a similar way.
\end{proof}

\section{Conharmonically Flat Doubly Warped Product Manifolds}

Let $(M, g)$ be an $m-$dimensional Riemannian manifold, $m\geq 3$. The \emph{con\-har\-mon\-ic curvature tensor} $\mathcal{H}$ \cite{is} is defined by
\begin{equation*}\label{e29}
\mathcal{H}_{XY}Z=R_{XY}Z-\frac{1}{m-2}\big(g(Y,Z)QX-g(X,Z)QY+\Ric(Y,Z)X-\Ric(X,Z)Y\big),
\end{equation*}
where $R$ is the Riemann curvature, $\Ric$ is the Ricci curvature and $Q$ is the Ricci operator. If the conharmonic curvature $\mathcal{H}$ vanishes identically on $M$, then $M$ is called \emph{conharmonically flat}.\\

By using Lemma \ref{L1} and \eqref{e124}--\eqref{e144}, we obtain the following components of the conharmonic curvature tensor on the factors of a doubly warped product.
\begin{lemma} \label{L4}
Let $(M=_{{f}_{2}}\!M_1 \!\times_{{f}_{1}}\!M_2,g)$ be a doubly warped product manifold. Then,
we have
\begin{align}
\mathcal{H}_{XY}Z&=^{1}\!R_{XY} Z+g(X,Z)H^l(Y)-g(Y,Z)H^l(X) \nonumber\\
&\quad-\displaystyle\frac{1}{m-2}\Big(\frac{1}{f_2^2}g(Y,Z)\Big(^1\!QX-\frac{m_2}{f_1}H^{f_1}(X)-f_2^2 \Delta l \cdot X\Big)  \nonumber\\
&\quad-\frac{1}{f_2^2}g(X,Z)\Big(^1\!QY-\frac{m_2}{f_1}H^{f_1}(Y)-f_2^2 \Delta l \cdot Y\Big) \nonumber \\
&\quad+^1\!\Ric(Y,Z)X-\frac{m_2}{f_1}h_1^{f_1}(Y,Z)X-\Delta l g(Y,Z)X   \nonumber\\
&\quad-^1\!\Ric(X,Z)Y+\frac{m_2}{f_1}h_1^{f_1}(X,Z)Y+\Delta l g(X,Z)Y\Big), \label{e234}
\end{align}
\begin{align}
\mathcal{H}_{UV}W&=^{2}\!R_{UV} W+g(U,W)H^k(V)-g(V,W)H^k(U) \nonumber\\
&\quad-\displaystyle\frac{1}{m-2}\Big(\frac{1}{f_1^2}g(V,W)\Big(^2\!QU-\frac{m_1}{f_2}H^{f_2}(U)-f_1^2 \Delta k \cdot U\Big) \nonumber\\
&\quad-\frac{1}{f_1^2}g(U,W)\Big(^2\!QV-\frac{m_1}{f_2}H^{f_2}(V)-f_1^2\Delta k \cdot V\Big) \nonumber \\
&\quad+^2\!\Ric(V,W)U-\frac{m_1}{f_2}h_2^{f_2}(V,W)U-\Delta k g(V,W)U \nonumber \\
&\quad-^2\!\Ric(U,W)V+\frac{m_1}{f_2}h_2^{f_2}(U,W)V+\Delta k g(U,W)V\Big). \label{e284}
\end{align}
\end{lemma}
\begin{theorem}
Let $(M=_{{f}_{2}}\!M_1 \!\times_{{f}_{1}}\!M_2,g)$ be a doubly warped product manifold with $m_1,m_2>1$.
If $M$ is conharmonically flat, then \\

\item \textbf{(a)} \quad $(\nabla^{1}f_{1}, \lambda_{1})$ defines a gradient $f$-almost Ricci soliton on $(M_1, g_1)$,\\
\\
where $\lambda_{1}=\frac{f^{2}_{2}}{m-m_1}\big(\frac{^1\!\tau}{f_2^2}-\frac{m_2}{f_1f_2^2}\Delta_1f_1+(m_1-1)\big((m-2)g(\nabla l,\nabla l)-2\Delta l\big)\big)$
and $f=-\frac{m_2(1-(m_1-1)f_2^2)}{(m-m_1)f_1f_2^2}$;\\

\item \textbf{(b)}\quad $(\nabla^{2}f_{2}, \lambda_{2})$ defines a gradient $f$-almost Ricci soliton on $(M_2, g_2)$,\\
where $\lambda_{2}=\frac{f^{2}_{1}}{m-m_2}\big(\frac{^2\!\tau}{f_1^2}-\frac{m_1}{f_1^2f_2}\Delta_2f_2+(m_2-1)\big((m-2)g(\nabla k,\nabla k)-2\Delta k\big)\big)$
and $f=-\frac{m_1(1-(m_2-1)f_1^2)}{(m-m_2)f_1^2f_2}$.
\end{theorem}
\begin{proof} Let $(M=_{{f}_{2}}\!M_1 \!\times_{{f}_{1}}\!M_2,g)$ be a conharmonically flat doubly warped product manifold with $m_1,m_2>1$.
Then, for any $X,Y, Z\in\mathfrak{L}(M_{1})$, we have
\begin{align}
&^{1}\!R_{XY} Z+g(X,Z)H^l(Y)-g(Y,Z)H^l(X)\nonumber\\
&-\displaystyle\frac{1}{m-2}\Big(\frac{1}{f_2^2}g(Y,Z)\Big(^1\!QX-\frac{m_2}{f_1}H^{f_1}(X)-f_2^2 \Delta l \cdot X\Big)\nonumber\\
&-\frac{1}{f_2^2}g(X,Z)\Big(^1\!QY-\frac{m_2}{f_1}H^{f_1}(Y)-f_2^2 \Delta l \cdot Y\Big) \nonumber\\
&+^1\!\Ric(Y,Z)X-\frac{m_2}{f_1}h_1^{f_1}(Y,Z)X-\Delta l g(Y,Z)X \nonumber \\
&-^1\!\Ric(X,Z)Y+\frac{m_2}{f_1}h_1^{f_1}(X,Z)Y+\Delta l g(X,Z)Y\Big)=0, \nonumber
\end{align}
from \eqref{e234}. Hence, for any $T \in\mathfrak{L}(M_{1}),$ using \eqref{e0} and \eqref{0e4}
$$g_{1}(H^l(X), T)=g(\nabla l, \nabla l)g_{1}(X, T),$$
we obtain
\begin{align}
&g_{1}(^{1}\!R_{XY}Z, T) =-g(X,Z)g_1(H^l(Y),T)+g(Y,Z)g_1(H^l(X),T)\nonumber\\
&\quad+\displaystyle\frac{1}{m-2}\Big(\frac{1}{f_2^2}g(Y,Z)\Big(g_1(^1\!QX,T)-\frac{m_2}{f_1}g_1(H^{f_1}(X),T)-f_2^2 \Delta l g_1(X,T)\Big)\nonumber\\
&\quad-\frac{1}{f_2^2}g(X,Z)\Big(g_1(^1\!QY,T)-\frac{m_2}{f_1}g_1(H^{f_1}(Y),T)-f_2^2 \Delta l g_1(Y,T)\Big) \nonumber\\
&\quad+^1\!\Ric(Y,Z)g_1(X,T)-\frac{m_2}{f_1}h_1^{f_1}(Y,Z)g_1(X,T)-\Delta l g(Y,Z)g_1(X,T) \nonumber \\
&\quad-^1\!\Ric(X,Z)g_1(Y,T)+\frac{m_2}{f_1}h_1^{f_1}(X,Z)g_1(Y,T)+\Delta l g(X,Z)g_1(Y,T)\Big). \nonumber
\end{align}
By contracting this equation over $X$ and $T$, we get
\begin{align}
&^1\!\Ric(Y, Z)=\frac{f^{2}_{2}}{m-m_1}\Big(\frac{^1\!\tau}{f_2^2}-\frac{m_2}{f_1f_2^2}\Delta_1f_1\nonumber\\
&\quad +(m_1-1)\big((m-2)g(\nabla l,\nabla l)-2\Delta l\big)\Big)g_{1}(Y, Z)
+\frac{m_2(1-(m_1-1)f_2^2)}{(m-m_1)f_1f_2^2}h_1^{f_1},\nonumber
\end{align}
which proves \textbf{(a)}. The assertion \textbf{(b)} can be obtained in a similar way.
\end{proof}

\section*{Acknowledgment}
The second author was supported by 1001-Scientific and Technological Research Projects Funding Program of The Scientific and Technological Research Council of Turkey (TUBITAK) with project number 119F179.

\vspace{10pt}
{\em Adara M. Blaga}, West University of Timi\c{s}oara, Department of Mathematics, Bld. V. P\^{a}rvan, 300223, Timi\c{s}oara, Rom\^{a}nia,
e-mail: \texttt{adarablaga@yahoo.com}

\vspace{8pt}
{\em Hakan M. Ta\c{s}tan}, \.{I}stanbul University, Faculty of Science, Department of Mathematics, Vezneciler, 34134, \.{I}stanbul, Turkey,
e-mail: \texttt{hakmete@istanbul.edu.tr}

\end{document}